\documentclass{amsart}
\usepackage{amsmath,amsfonts,amscd,amssymb,amsthm,epsfig,euscript,eufrak}
\usepackage{amsxtra,mathrsfs}
\usepackage{mathrsfs}
\usepackage{verbatim,epsf}
\newtheorem{theorem}{Theorem}
\newtheorem{proposition}{Proposition}
\newtheorem{lemma}{Lemma}
\newtheorem{corollary}{Corollary}
\theoremstyle{definition}
\newtheorem{definition}[theorem]{Definition}

\theoremstyle{remark}
\newtheorem{remark}{Remark}
\usepackage{pdfsync}
\newcommand{\curly}[1]{\mathscr{#1}}
\newcommand{\cE}{\curly{E}}

\newcommand{\cs}{\mathrm{cs}}
\newcommand{\Str}{\mathrm{Struct}}
\newcommand{\CS}{\mathrm{CS}}
\newcommand{\ch}{\mathrm{ch}}
\newcommand{\del}{\partial}
\newcommand{\s}{\sigma}
\newcommand{\f}{\bar{f}}
\newcommand{\bc}{\mathrm{bc}}   
\newcommand{\BC}{\mathrm{BC}}
\newcommand{\im}{\mathrm{Im}}
\newcommand{\re}{\mathrm{Re}}
\newcommand{\A}{\mathcal{A}}

\newcommand{\delb}{\bar{\partial}}
\newcommand{\C}{{\mathbb{C}}}
\newcommand{\R}{{\mathbb{R}}}
\newcommand{\PP}{{\mathbb{P}}}

\DeclareMathOperator{\Tr}{tr}
\DeclareMathOperator{\td}{td}
\numberwithin{equation}{section}
\usepackage{hyperref}
\hypersetup{colorlinks,citecolor=blue,plainpages=false,hypertexnames=false}
\begin{document}
\title{On Bott-Chern forms and their applications}
\author{Vamsi P. Pingali}
\address{Department of Mathematics, Johns Hopkins University, 404 Krieger Hall, Baltimore, MD - 21218, USA}
\email{vpingal1@jhu.edu}    
\author{Leon A. Takhtajan}
 \address{Department of Mathematics, Stony Brook University, Stony Brook NY, 11794-3651, USA;
 \newline
            The Euler Mathematical Institute, Pesochnaya nab. 10, St. Petersburg 197022, Russia}
\email{leontak@math.sunysb.edu}
\maketitle
\begin{abstract}We use Chern-Weil theory 
for Hermitian holomorphic vector bundles with canonical connections for explicit computation of the Chern forms of trivial bundles with special non-diagonal Hermitian metrics. We prove that every $\delb\del$-exact  real form of the type $(k,k)$ on an $n$-dimensional  complex manifold $X$ arises as a difference of the Chern character forms of trivial Hermitian vector bundles with canonical connections, and that modulo $\im\,\del+\im\,\delb$ every real form of type $(k,k)$, $k<n$, arises as a Bott-Chern form for two Hermitian metrics on some trivial vector bundle over $X$. The latter result is a complex manifold analogue of Proposition 2.6 in the paper \cite{SS} by J. Simons and D. Sullivan. As an application, we obtain an explicit formula for the Bott-Chern form of a short exact sequence of holomorphic vector bundles considered by Bott and Chern in \cite{BC}, for the case when the first term is a line bundle.
\end{abstract}
\section{Introduction}\label{Sect:1}

In \cite{SS} J. Simons and D. Sullivan constructed a simple geometric model
for differential $K$-theory. The model uses the notion of a structured vector bundle: a pair $(V,\{\nabla\})$ consisting of a complex vector bundle $V$ over a smooth manifold $X$ and the equivalence class of a connection $\nabla$. Two connections $\nabla^{0}$ and $\nabla^{1}$ are said to be equivalent if the corresponding Chern-Simons differential form is exact. The main technical innovation in \cite{SS} was Proposition 2.6 which states that all odd forms on $X$, modulo some natural relations, arise as the Chern-Simons forms between the trivial connection and an arbitrary connection on trivial bundles over $X$. It allows one to prove that differential $K$-theory has a natural analogue of the celebrated character diagram for the ring of Cheeger-Simons differential characters (see \cite{SS2} and \cite{Cheeger}). 

For Hermitian holomorphic vector bundles --- holomorphic vector bundles over the complex manifold $X$ with Hermitian metrics --- analogues of the Chern-Simons forms are the Bott-Chern forms, which were introduced in \cite{BC} earlier than the Chern-Simons forms in \cite{CS}. The corresponding differential $K$-theory was defined by H. Gillet and C. Soul\'{e} in \cite{GS}. 

In this paper we use Chern-Weil theory for Hermitian holomorphic vector bundles with canonical connections for explicit computation of the Chern forms for trivial bundles with special non-diagonal Hermitian metrics. Our first result is the exact analogue of Proposition 2.6 in \cite{SS} for complex manifolds. Namely, we prove that all real forms of type $(k,k)$ on an $n$-dimensional complex manifold $X$, $k<n$,  modulo $\im\,\del+\im\,\delb$, arise as Bott-Chern forms for Hermitian metrics on trivial vector bundles over $X$.
As in the smooth manifold case, we deduce this statement from Theorem \ref{Venice-hol}, which says that every $\delb\del$-exact real form of type $(k,k)$ on a complex manifold $X$ arises as a difference of the Chern character forms on trivial Hermitian vector bundles. The proof uses an explicit computation of Chern forms for trivial vector bundles, given in Lemma \ref{Main}.  The actual computation is based on Lemma \ref{Algebra}, 
which gives an explicit formula for determinants of special matrices over a ring with nilpotents. These results have interesting applications of their own. Thus using Lemma \ref{Algebra},
in Proposition \ref{B-C-formula-1} we obtain an explicit formula for the Bott-Chern form of a short exact sequence of holomorphic vector bundles considered by Bott and Chern
in \cite{BC}, for the case when the first term is a line bundle. 

Here is a more detailed content of the paper. In Section \ref{Sect:2}, for the convenience of the reader, we give a brief review of \cite{SS}. Namely, we use the definition of
the Chern-Simons forms inspired by the approach of H. Gillet and C. Soul\'{e} for the complex manifold case \cite{GS}, and deduce a somewhat stronger analogue of Proposition 2.6 in \cite{SS} --- Corollary \ref{V-real} --- from Proposition \ref{Venice}. The latter states that for every exact even form $\omega$ on a smooth manifold $X$ there is a trivial vector bundle $V$ over $X$ with a connection $\nabla$ such that
$$\ch(V,\nabla)-\ch(V,d)=\omega,$$
where $d$ stands for the trivial connection on $V$. We use Corollary \ref{V-real} to give a different proof of Theorem 1.15 in \cite{SS}, that does not rely on the existence of universal connections and works for the non-compact case as well.

In Section \ref{Sect:3} we prove Theorem \ref{Venice-hol}, which states that for every $\delb\del$-exact real form $\omega$ of type $(k,k)$ on a general complex manifold $X$ 
there is a trivial vector bundle $E$ over $X$ with two Hermitian metrics $h_{1}$ and $h_{2}$ such that
$$\ch(E,h_{1})-\ch(E,h_{2})=\omega.$$ 
The proof in the general case is based on Lemma \ref{Main}, where we explicitly compute the Chern form of a trivial vector bundle over $X$ of arbitrary rank equipped with a special non-diagonal Hermitian metric, and on Lemma \ref{Connes-hol} where we express real forms of type $(k,k)$ as finite linear combinations of wedge products of real $(1,1)$-forms of special type. In turn, Lemma \ref{Main} is based on the linear algebra Lemma \ref{Algebra}, which gives an explicit formula for the determinants of certain matrices over a ring with nilpotents. When $X$ is compact or is a submanifold of
$\C^{n}$, we give another proof of Theorem \ref{Venice-hol} based on Lemma \ref{Connes-hol-2}.
We deduce the complex manifold analogue of Proposition 2.6 in \cite{SS} --- Corollary \ref{V-complex} --- by using the Gillet-Soul\'{e} definition of the Bott-Chern forms \cite{GS}. 

In Section \ref{Sect:4} we present some applications of our results. In Section \ref{Sub:4.1}, we use Corollary \ref{V-complex} in order to get rid of the differential form in the complex manifold version of differential $K$-theory \cite{GS}. However, developing differential $K$-theory for the complex manifolds in the spirit of \cite{SS} is an open and difficult problem since, in general, `inverses' for the holomorphic bundles do not exist. In Section \ref{Sub:4.2} we explicitly compute the Bott-Chern form for the short exact sequence of Hermitian holomorphic vector bundles considered by Bott and Chern in \cite{BC} for the case when the first term is a line bundle. This formula can be used to simplify the computation in \cite{Mou} of the Bott-Chern form for the metrized Euler sequence of a projectivized vector bundle.

\section{Complex vector bundles over a smooth manifold}\label{Sect:2}
\subsection{Chern-Simons secondary forms}\label{Sub:2.1}
Let $X$ be a smooth $n$-dimensional manifold, let
$$\mathcal{A}(X)=\bigoplus_{k=0}^{n}\mathcal{A}^{k}(X,\C)=\mathcal{A}^{\mathrm{even}}(X)\oplus\mathcal{A}^{\mathrm{odd}}(X)$$
be the graded commutative algebra of smooth complex differential forms on $X$, and let $V$ be a $C^{\infty}$-complex vector bundle over $X$ with a connection $\nabla$.  Recall that the Chern character form $\ch(V,\nabla)$ for the pair $(V,\nabla)$ is defined by
$$\ch(V,\nabla)=\Tr\left\{\exp\left(\frac{\sqrt{-1}}{2\pi}\,\nabla^{2}\right)\right\}\in \mathcal{A}^{\mathrm{even}}(X).$$
Here $\nabla^{2}$ is the curvature of the connection $\nabla$ --- an $\mathrm{End} \,V$-valued $2$-form on $X$ ---
and $\Tr$ is the trace in the endomorphism bundle $\mathrm{End} \,V$. The Chern character form is closed, $\mathrm{d}\,\ch(V,\nabla)=0$, and its cohomology class in $H^{\ast}(X,\C)$ does not depend on the choice of  $\nabla$.

Let $\nabla^{0}$ and $\nabla^{1}$ be two connections on $V$. In \cite{CS}, S.S. Chern and J. Simons introduced secondary characteristic forms --- 
the Chern-Simons forms. Namely, the Chern-Simons form $\cs(\nabla^{1},\nabla^{0})\in\mathcal{A}^{\mathrm{odd}}(X)$  defined modulo $d\mathcal{A}^{\mathrm{even}}(X)$, satisfies the equation 
\begin{equation} \label{cs-1}
d\,\cs(\nabla^{1},\nabla^{0})=\ch(V,\nabla^{1})-\ch(V,\nabla^{0}),
\end{equation} 
and enjoys a functoriality property under the pullbacks with smooth maps. 

Here we present a construction of the Chern-Simons form $\cs(\nabla^{1},\nabla^{0})$ which is similar to the construction of Bott-Chern forms for holomorphic vector bundles given by H. Gillet and C. Soul\'{e} in \cite{GS}. Namely, for a given $V$ put $\tilde{V}=\pi^{*}(V)$, where $\pi: X\times S^{1}\mapsto X$ is a projection, and $S^{1}=\{e^{i\theta}, 0\leq\theta< 2\pi\}$. For every $\theta$ define the map $i_{\theta}:X\mapsto X\times S^{1}$ by $i_{\theta}(x)=(x,e^{i\theta})$, and let $\tilde{\nabla}$ be a connection on $\tilde{V}$ such that 
$$i_{0}^{*}(\tilde{\nabla})=\nabla^{0},\quad i_{\pi}^{*}(\tilde{\nabla})=\nabla^{1}.$$
Denote by $g$ a function defined by
$$g(\theta)=\begin{cases} 0\quad \text{if} \quad 0\leq \theta<\pi, \\
1 \quad \text{if} \quad \pi\leq\theta<2\pi,\end{cases}$$
and extended $2\pi$-periodically to $\mathbb{R}$. It defines a function $g: S^{1}\mapsto\mathbb{R}$, which is discontinuous at $0$ and $\pi$.\\
\indent In this construction the Chern-Simons form is
\begin{equation} \label{cs-def}
\cs(\nabla^{1},\nabla^{0})=\pi_{\ast}(g(\theta)\ch(\tilde{V},\tilde{\nabla}))=\int_{S^{1}}g(\theta)\ch(\tilde{V},\tilde{\nabla})\in \mathcal{A}^{\mathrm{odd}}(X)
\end{equation} 
-- integration over the fibres of $\pi$.

\begin{remark} Connection $\tilde{\nabla}$ is trivial to construct. If in local coordinates $\nabla^{i}=d_{x}+A^{i}(x)$, where $d_{x}$ is deRham differential on $X$ and $i=0,1$, then 
$$\tilde{\nabla}=d_{x}+d_{\theta}+A(x,\theta),$$
where $A(x,\theta)$ is $2\pi$-periodic and $A(x,0)=A^{0}(x)$, $A(x,\pi)=A^{1}(x)$.
\end{remark}
The following lemma is proved using exactly the same technique as in \cite{GS}.
\begin{lemma} \label{C-S-1}The Chern-Simons form  $\cs(\nabla^{1},\nabla^{0})$
satisfies the equation \eqref{cs-1}, and modulo $d\mathcal{A}^{\mathrm{even}}(X)$ it does not depend on the choice of connection $\tilde{\nabla}$.
\end{lemma}

\begin{definition} Put
$$\CS(\nabla^{1},\nabla^{0})=\cs(\nabla^{1},\nabla^{0})\!\!\!\!\!\mod\!d\mathcal{A}^{\mathrm{even}}(X),$$
which, according to Lemma \ref{C-S-1}, is a well-defined element in $\widetilde{\mathcal{A}}^{\mathrm{odd}}(X)=\mathcal{A}^{\mathrm{odd}}(X)/d\mathcal{A}^{\mathrm{even}}(X)$.
\end{definition}
\begin{remark} Formula \eqref{cs-def} can be written as 
$$\cs(\nabla^{1},\nabla^{0})=\int_{\pi}^{2\pi}\ch(\tilde{\nabla}),$$
and the choice of points $\pi$ and $2\pi$ on the unit circle is immaterial. Using the change of variables, for every $\alpha<\beta$ on $S^{1}$ we get 
\begin{equation} \label{cs-2}
\cs(\nabla^{1},\nabla^{0})=\int_{\alpha}^{\beta}\ch(\tilde{\nabla}),
\end{equation}
where now $i_{\alpha}^{*}(\tilde{\nabla})=\nabla^{0}$, $i_{\beta}^{*}(\tilde{\nabla})=\nabla^{1}$.
\end{remark}
Using \eqref{cs-2} and Lemma \ref{C-S-1}, we immediately get
\begin{corollary}\label{Cor-1}
$$\CS(\nabla^{2},\nabla^{0})=\CS(\nabla^{2},\nabla^{1})+\CS(\nabla^{1},\nabla^{0}).$$
\end{corollary}

\subsection{$K$-theory}\label{Sub:2.2}
Let $K_{0}(X)$ be the Grothendieck group of $X$ --- the quotient of the free abelian group generated by the isomorphism 
classes $[V]$ of complex vector bundles $V$ over $X$ by the relations $[V] + [W]=[V\oplus W]$. 
In \cite{SS}, the authors defined a version of differential $K$-theory using the notion of a \emph{structured bundle}. Namely, connections $\nabla^{0}$ and $\nabla^{1}$ on a complex vector bundle $V$ over $X$ are called equivalent, if $\CS(\nabla^{1},\nabla^{0})=0$. It follows from Corollary \ref{Cor-1} that it is an equivalence relation.
\begin{definition} A pair $\mathcal{V}=(V,\{\nabla\})$, where $\{\nabla\}$ is an equivalence class of connections on $V$, is called a \emph{structured bundle}.
\end{definition}
Denote by $\Str(X)$ the set of all equivalence classes of structured bundles over $X$. It is shown in \cite{SS} that $\Str(X)$ is a commutative semi-ring with respect to the direct sum
$\oplus$ and tensor product $\otimes$ operations, and we denote by $\hat{K}_{0}(X)$ the corresponding Grothendieck ring.

We have two natural ring homomorphisms: the `forgetful map' 
$$\delta: \hat{K}_{0}(X)\rightarrow K_{0}(X),$$ 
given by $[\mathcal{V}]\mapsto [V]$ for $\mathcal{V}=(V,\{\nabla\})$,
and the Chern character map 
$$\ch: \hat{K}_{0}(X)\rightarrow \A^{\mathrm{even}}(X),$$
given by $[\mathcal{V}]\mapsto\ch(V,\nabla)$. For a trivial bundle $V$ with trivial connection $\nabla=d$, $\ch(V,d)=\mathrm{rk}(V)$ --- the rank of $V$. The mapping $\delta$ is surjective and for compact $X$ its kernel consists of all differences $[\mathcal{V}]-[\mathcal{F}]$ such that
$\mathcal{V}=(V,\{\nabla\})$, where $V$ is stably trivial, $V\oplus M=N$
for some trivial bundles $M$ and $N$, and $\mathcal{F}=(F,\{\nabla^{F}\})$, where $F$ is trivial and $\mathrm{rk}(F)=\mathrm{rk}(V)$. It is an outstanding problem to describe the image of the Chern character map.

The following result is crucial for our approach to the differential $K$-theory of Simons-Sullivan \cite{SS}. 
\begin{proposition} \label{Venice} The image of the Chern character map contains all exact forms. Specifically, for every exact even form $\omega$ there is a trivial vector bundle $V=X\times\C^{r}$ with a connection $\nabla=d+A$ such that
$$\ch(V,\nabla)-\ch(V,d)=\omega.$$
\end{proposition}
The proof is based on the following simple fact (see, e.g., \cite[p. 16]{DR}), which is proved by partition of unity, or by using the Whitney embedding theorem.
\begin{lemma} \label{Connes} Every $\eta\in\mathcal{A}^{k}(X)$ can be represented as a finite sum of the basic forms $f_{1}df_{2}\wedge\dots\wedge df_{k+1}$, where
$f_{1},\dots,f_{k+1}$ are smooth functions on $X$. If the form $\eta$ is real, one can choose the basic forms such that all functions $f_{i}$ are real-valued, and if $\eta$ is zero on an open $U\Subset X$, there is a representation such that all functions $f_{i}$ vanish on $U$. 
\end{lemma}
\begin{proof}[Proof of Proposition \ref{Venice}] Induction by the degree in $d \A^{\mathrm{odd}}(X)\subset \A^{\mathrm{even}}(X)$. According to Lemma \ref{Connes}, it is sufficient to consider only basic forms in $\A^{\mathrm{odd}}(X)$.

For a basic $1$-form $\alpha=f_{1}df_{2}$ we have $\omega=d\alpha=df_{1}\wedge df_{2}$, so that
$$\ch(L,\nabla)-\ch(L,d)=\ch(L,\nabla)-1=\omega,$$
where $L$ is a trivial line bundle over $X$ with connection $\nabla=d-2\pi\sqrt{-1}f_{1}df_{2}$ and curvature $\nabla^2 = -2\pi \sqrt{-1} df_1 \wedge df_2$. 

Now suppose that all exact forms of degree $\leq 2k$ are in the image of $\ch$. 
For a basic $(2k+1)$-form $\alpha=f_{1}df_{2}\wedge\cdots\wedge df_{2k+2}$ we have $\omega =d\alpha=df_{1}\wedge df_{2}\wedge\cdots\wedge df_{2k+2}$, which can be also written as
$$\omega =\frac{1}{(k+1)!}(df_{1}\wedge df_{2}+\dots+df_{2k+1}\wedge df_{2k+2})^{k+1}.$$
Let $V$ be a trivial line bundle over $X$ with 
$$\nabla=d -2\pi\sqrt{-1}(f_{1}df_{2}+\dots + f_{2k+1}df_{2k+2}),$$
so that
$$\nabla^{2}=-2\pi\sqrt{-1}(df_{1}\wedge df_{2}+\dots + df_{2k+1}\wedge df_{2k+2}).$$
Then $\ch(V,\nabla)-1-\omega $ is an exact form of degree $\leq 2k$, and by induction is in the image of $\ch$.
\end{proof}
\begin{remark} If the form $\omega$ is real then the connection $\nabla$ in Proposition \ref{Venice} is compatible with the metric on $V$ given by the standard
Hermitian metric on $\C^{r}$.
\end{remark}
\begin{remark} It immediately follows from the second statement of Lemma \ref{Connes} and the proof of Proposition \ref{Venice} that if form $\omega$ vanishes on open $U\subset X$, then the connection $\nabla=d+A$ can be chosen such that $A=0$ on $U$.
\end{remark}
\begin{corollary} \label{V-real} For every $\alpha\in\widetilde{\mathcal{A}}^{\mathrm{odd}}(X)$ there is a trivial vector bundle $V$ with connection $\nabla$
such that $\CS(\nabla,d)=\alpha$.
\end{corollary}
\begin{proof} For the given $\alpha\in\mathcal{A}^{\mathrm{odd}}(X)$ let $\Theta\in\mathcal{A}^{\mathrm{odd}}(X\times S^{1})$ be such that under the inclusion map $i_{\theta}: X\rightarrow X\times S^{1}$ one has  $i_{\pi}^{\ast}(\Theta)=\alpha$ and $i_{\theta}^{\ast}(\Theta)=0$ for all $\theta$ in some neighborhood of $0$. Applying Proposition \ref{Venice} to the manifold $X\times S^{1}$ and the exact even form $-d\Theta$, we have
$$\ch(\tilde{V},\tilde\nabla)-\mathrm{rk}(\tilde{V})=-(d_{x}+d_{\theta})\Theta.$$
Integrating over $\theta$ from $\pi$ to $2\pi$ we get
$$\alpha=\cs(\nabla^{1},\nabla^{0}) +d_{x}\int_{\pi}^{2\pi}\Theta,$$
for connections $\nabla^{0}=i_{0}^{\ast}(\tilde{\nabla})$ and $\nabla^{1}=i_{\pi}^{\ast}(\tilde\nabla)$ on a trivial bundle $V$ --- a pullback of the trivial bundle $\tilde{V}$ to $X$.  Finally, it follows from Remark 2.4 that one can choose the connection $\tilde\nabla$ on $\tilde{V}$ such that $i_{0}^{\ast}(\tilde{\nabla})=d$. Thus putting $\nabla=\nabla^{1}$ we obtain $\CS(\nabla,d)=\alpha\!\!\!\mod\!d\mathcal{A}^{\mathrm{even}}(X)$.
\end{proof}
\begin{remark} Corollary \ref{V-real} gives a somewhat stronger form of Proposition 2.6 in \cite{SS}, the so-called ``Venice lemma'' of J. Simons\footnote{D. Sullivan, private communication.}.
\end{remark}

Corollary \ref{V-real} can be used to give a different proof of Theorem 1.15 in \cite{SS}, that does not rely on the existence of universal connections and works for the non-compact case as well.

Namely, following \cite{SS}, let
$$\Str_{\mathrm{ST}}(X)=\{[\mathcal{V}]=[(V,\{\nabla\})]\in\Str(X)\,|\,\text{$V$ is stably trivial}\}$$ 
be the stably trivial sub-semigroup of $\Str(X)$, and for $[\mathcal{V}]\in\Str_{\mathrm{ST}}(X)$ define
$$\widehat{\CS}([\mathcal{V}])=\CS(\nabla^{N},\nabla\oplus\nabla^{F})\in\mathcal{A}^{\mathrm{odd}}(X)/d\mathcal{A}^{\mathrm{even}}(X),$$
where $V\oplus F=N$ with trivial bundles $F$ and $N$, and $\nabla^{F}$, $\nabla^{N}$ are flat connections on these bundles. According to Proposition 2.4 in \cite{SS}, for another choice of trivial bundles $\bar{F}$ and $\bar{N}$ with flat connections $\nabla^{\bar{F}}$, $\nabla^{\bar{N}}$ we have
$$\CS(\nabla^{N},\nabla\oplus\nabla^{F})-\CS(\nabla^{\bar{N}},\nabla\oplus\nabla^{\bar{F}})\in\mathcal{T}(X),$$
where $\mathcal{T}(X)$ is a subgroup in $\widetilde{\mathcal{A}}^{\mathrm{odd}}(X)$ consisting of $\CS(\nabla,\nabla')$
for all trivial bundles $F$ and flat connections $\nabla$, $\nabla'$ on $F$.
Therefore, the mapping $\widehat{\CS}: \Str_{\mathrm{ST}}(X)\rightarrow \widetilde{\mathcal{A}}^{\mathrm{odd}}(X)
/\mathcal{T}(X)$ is a well-defined homomorphism of semigroups. 

According to Corollary \ref{V-real} the map $\widehat{\CS}$ is surjective, and according to Proposition 2.5 in \cite{SS}, $\ker\widehat{\CS}=\Str_{\mathrm{SF}}(X)$ --- the subgroup of \emph{stably flat} structured bundles. By definition, $[\mathcal{V}]\in\Str_{\mathrm{ST}}(X)$ is stably flat if 
$$\mathcal{V}\oplus\mathcal{F}=\mathcal{N},$$
where
$\mathcal{F}=(F,\{\nabla^{F}\})$ and $\mathcal{N}=(N,\{\nabla^{N}\})$ are trivial bundles with equivalence classes of flat connections. Since map $\widehat{\CS}$ is onto and $\widetilde{\mathcal{A}}^{\mathrm{odd}}(X)/\mathcal{T}(X)$ is a group,
for every $[\mathcal{V}]\in\Str_{\mathrm{ST}}(X)$ there is $[\mathcal{W}]\in\Str_{\mathrm{ST}}(X)$ such that $[\mathcal{V}]+[\mathcal{W}]\in\Str_{\mathrm{SF}}(X)$. This introduces a group structure on the coset space $\Str_{\mathrm{ST}}(X)/\Str_{\mathrm{SF}}(X)$, and we arrive at the following statement.
\begin{proposition} \label{group} The map $\widehat{\CS}$ induces a group isomorphism
$$\widehat{\CS}:\Str_{\mathrm{ST}}(X)/\Str_{\mathrm{SF}}(X)\rightarrow \widetilde{\mathcal{A}}^{\mathrm{odd}}(X)/\mathcal{T}(X).$$
\end{proposition}

We obtain Theorem 1.15 in \cite{SS} as an immediate corollary of this result.
\begin{corollary} \label{C-S-nc} Every structured bundle over $X$ has a structured inverse, i.e., 
for every $[\mathcal{V}]=[(V,\{\nabla\})]\in\Str(X)$ there is $[\mathcal{W}]=[(W,\{\nabla^{W}\})]\in\Str(X)$ such that
$$[\mathcal{V}]+[\mathcal{W}]=[\mathcal{N}],$$
where $\mathcal{N}=(N,\{\nabla^{N}\})$ is a trivial bundle with flat connection.
\end{corollary} 
\begin{proof} For $[\mathcal{V}]=[(V,\{\nabla\})]\in\Str(X)$ let $U$ be such that $V\oplus U=F$ --- a trivial bundle over $X$.
Then $[\mathcal{F}]=[(F,\{\nabla\oplus\nabla^{U}\})]\in\Str_{\mathrm{ST}}(X)$ for any choice of connection $\nabla^{U}$ on $U$. By Proposition \ref{group}, there exists $[\mathcal{H}]=[(H,\{\nabla^{H}\})]\in\Str_{\mathrm{ST}}(X)$
such that $[\mathcal{F}]+[\mathcal{H}]\in\Str_{\mathrm{SF}}(X)$, i.e., there are trivial bundles
$M$ and $N$ with flat connections $\nabla^{M}$ and $\nabla^{N}$ such that
$\mathcal{F}\oplus\mathcal{H}\oplus\mathcal{M}=\mathcal{N}$. Putting 
$$\mathcal{W}=(U\oplus H\oplus M,\{\nabla^{U}\oplus\nabla^{H}\oplus\nabla^{M}\}),$$
we obtain $[\mathcal{V}]+[\mathcal{W}]=[\mathcal{N}]$.
\end{proof}
Using Corollary \ref{C-S-nc} we conclude that all results in \cite{SS} hold for the non-compact case as well. 
\section{Holomorphic vector bundles over a complex manifold}\label{Sect:3}
\subsection{Chern-Weil theory}\label{Sub:3.1}
Let $(E,h)$ be a holomorphic Hermitian vector bundle --- a holomorphic vector bundle of rank $r$ over a complex manifold $X$, $\dim_{\C}X=n$, with a Hermitian metric $h$. In what follows we always use a local trivialization of $E$ ---   an open cover $\{U_{\alpha}\}_{\alpha\in A}$ of $X$ and holomorphic transition functions $g_{\alpha \beta}:U_{\alpha}\cap U_{\beta}\rightarrow\mathrm{GL}(r,\C)$, satisfying the cocycle condition.
 In these terms, a Hermitian metric $h$ on $E$ is the collection $h=\{h_{\alpha}\}_{\alpha\in A}$, where $h_{\alpha}$ are positive-definite Hermitian $r\times r$ matrix-valued functions on $U_{\alpha}$, satisfying
$$h_{\beta}= g_{\alpha\beta}^{\ast}h_{\alpha}g_{\alpha\beta}\quad\text{on}\quad U_{\alpha}\cap U_{\beta},$$
and $g^{\ast}$ stands for the Hermitian conjugation.

Denote by $\nabla$ the canonical connection on the holomorphic Hermitian bundle $(E,h)$. In terms of a local trivialization it is given by the collection $\nabla=\{\nabla_{\alpha}\}_{\alpha\in A}$,
$$\nabla_{\alpha}=d+A_{\alpha}=\del+\delb+A_{\alpha}^{1,0}+A_{\alpha}^{0,1},$$
where $A_{\alpha}^{0,1}=0$ and $A_{\alpha}^{1,0}=h_{\alpha}^{-1}\del h_{\alpha}$ are $r\times r$ matrix-valued $(1,0)$-forms on $U_{\alpha}$, satisfying
$$A_{\beta}=g_{\alpha\beta}^{-1}A_{\alpha}g_{\alpha\beta} + g_{\alpha\beta}^{-1}\del g_{\alpha\beta}\quad\text{on}\quad U_{\alpha}\cap U_{\beta}.$$

The curvature of the canonical connection $\nabla=d+A$ is a collection $\Theta=\{\Theta_{\alpha}\}_{\alpha\in A}$, where $\Theta_{\alpha}=\delb A_{\alpha}$
are $r\times r$ matrix-valued $(1,1)$-forms on $U_{\alpha}$, satisfying
\begin{equation} \label{inv}
\Theta_{\beta}=g_{\alpha\beta}^{-1}\Theta_{\alpha}g_{\alpha\beta}\quad\text{on}\quad U_{\alpha}\cap U_{\beta}.
\end{equation}
Chern-Weil theory associates to any polynomial $\Phi$ on GL$(r,\mathbb{C})$, invariant under conjugation, a collection $\{\Phi(\Theta_{\alpha})\}_{\alpha\in A}$ which, according to \eqref{inv},
defines a global differential form $\Phi(\Theta)$ on $X$. Special cases of this construction are the total Chern form $c(E,h)$, the Chern character form $\ch(E,h)$, and the Todd form $\td(E,h)$ of a holomorphic Hermitian vector bundle $(E,h)$, which are respectively defined by 
\begin{align*}
c(E,h) & = \det\left(I + \frac{\sqrt{-1}}{2\pi}\,\Theta\right)= \sum_{k=0}^{r}c_{k}(E,h),\\
\ch(E,h) & =\Tr\left\{\exp\left(\frac{\sqrt{-1}}{2\pi}\,\Theta\right)\right\}=\sum_{k=0}^{n}\ch_{k}(E,h),\\
\intertext{and}
\td(E,h) & = \det \left\{ \dfrac{\dfrac{\sqrt{-1}}{2\pi}\Theta}{1-\exp\left(-\dfrac{\sqrt{-1}}{2\pi}\Theta\right)} \right\}=1+\sum_{k=1}^{r}\td_{k}(E,h).
\end{align*}
These differential forms are $\del$ and $\delb$ closed. 

For a holomorphic Hermitian bundle $(E,h)$ denote by $(E^{*},h_{*})$ the dual bundle $E^{*}$ with the induced metric $h_{*}=(h^{-1})^{t}$, and by $(\Lambda ^l E^{*}, \wedge^{l}h_{*})$ ---  corresponding $l$-th exterior powers of the bundle $(E^{*},h_{*})$ with Hermitian metrics $\wedge^{l} h_{*}$, induced by the metric $h_{*}$. The following formula provides a relation between Chern character forms and Todd forms, which promotes the well-known result for the cohomology classes (see \cite[Theorem 10.1.1]{Hirz}) to the level of forms.
\begin{lemma} \label{Hirzerbruch} Let $(E,h) $ be a holomorphic Hermitian vector bundle of rank $r$ with the metric $ h$. Then the following identity holds
$$\sum _{l=0} ^{r} (-1)^l \ch (\Lambda ^l E^{*},\wedge^{l}h_{*}) = \td(E,h)^{-1} c_r (E,h).$$
\end{lemma}
\begin{proof}
Consider the universal identity 
$$\det(I-A) = \sum _{l=0} ^{r} (-1)^l \Tr(\wedge^l A),$$ 
which holds for every $r \times r$ matrix $A$ over a commutative ring (see, e.g., \cite[p. 402]{GH}). Rewriting it as
$$\sum _{l=0} ^{r} (-1)^l \Tr(\wedge^l A) =\det A \det\left(\frac{I-A}{A}\right),$$ 
and replacing $A$ by $\exp \left(-\frac{\sqrt{-1}} {2\pi}\Theta\right)$, where $\Theta$ is the curvature of the canonical connection on $(E,h)$, we get
\begin{gather*}
 \sum _{l=0} ^{r} (-1)^l \Tr\left\{\Lambda^l \exp \left(-\frac{\sqrt{-1}} {2\pi}\Theta\right) \right\} =c_r (E,h) \td(E,h)^{-1}.
\end{gather*}
It remains to prove that $\ch (\Lambda ^l E ^{*})= \Tr \left( \Lambda ^l \exp \left(-\frac{\sqrt{-1}}{2\pi}\Theta\right) \right)$.
For this aim, consider the short exact sequence 
$$0\rightarrow \Lambda ^l E ^{*} \rightarrow \otimes ^l E ^{*} \rightarrow Q \rightarrow 0,$$ 
where $Q$ is the quotient vector bundle.  The bundle $\otimes^{l}E^{*}$ splits holomorphically into direct sum of orthogonal subbundles because there is a canonical section $\mathrm{Alt}: \otimes ^l E ^{*} \rightarrow \Lambda ^l E ^{*}$ which maps a tensor to its totally antisymmetric part. It induces a map  from $\mathrm{End}(\otimes ^l E ^{*})$ to $\mathrm{End}(\Lambda ^l E ^{*})$, which we continue to denote by  $\mathrm{Alt}$. We have $\Theta _{\Lambda ^l E ^{*}} = \mathrm{Alt} (\Theta _{\otimes ^l E ^{*}})$, which implies
\begin{gather*}
\ch (\Lambda ^l E^{*}) = \Tr \left\{\exp \left(\frac{\sqrt{-1}}{2\pi}\,\Theta _{\Lambda^l E^{*}}\right) \right\}\\
 =\Tr \left\{ \mathrm{Alt}\exp \left(-\frac{\sqrt{-1}}{2\pi}\sum_{k=1}^{l}\Theta_{k}\right)\right\}
 = \Tr \left\{\Lambda ^l \exp \left(-\frac{\sqrt{-1}}{2\pi}\Theta\right) \right\},
\end{gather*}  
where $\Theta_{k}=I\otimes\cdots\otimes I\otimes\Theta\otimes I\otimes\cdots\otimes I$, with $\Theta$ being the $k$-factor
of the $l$-fold tensor product.
\end{proof}

\subsection{Bott-Chern secondary forms}\label{Sub:3.2}
Let $h_{1}$ and $h_{2}$ be two Hermitian metrics on a holomorphic vector bundle $E$ over a complex manifold $X$. In the classic paper \cite{BC}, Bott and Chern showed the
existence of certain secondary characteristic forms. Their construction of these so-called Bott-Chern forms was generalized by Bismut, Gillet, and Soul\'{e} \cite{BGS}. The Bott-Chern form associated to an invariant polynomial $\Phi$ is an even differential form $\widetilde{\Phi}(E;h_{1},h_{2})\in\tilde{\A}(X,\C)=\A(X,\C)/(\im\,\del +\im\,\delb)$, satisfying
\begin{equation} \label{bc-def-1}
\Phi(E,h_{2}) - \Phi(E,h_{1}) =\frac{\sqrt{-1}}{2\pi}\,\delb\del\,\widetilde{\Phi}(E;h_{1},h_{2})
\end{equation}
and the functorial property 
\begin{equation} \label{bc-funct}
\widetilde{\Phi}(f^{*}(E),f^{*}(h_{1}),f^{*}(h_{2}))=f^{*}(\widetilde{\Phi}(E;h_{1},h_{2}))
\end{equation}
for holomorphic maps $f: Y\rightarrow X$ of complex manifolds. In \cite{BC} these forms were also defined for short exact sequences of hermitian holomorphic vector bundles. Namely, let $\cE$ 
$$ \begin{CD}
 0 @>>> F @>i>> E  @>p>> H  @>>> 0
 \end{CD} $$
be such an exact sequence, where holomorphic bundles $F, E$ and $H$ are equipped with Hermitian metrics $h_{F}, h_{E}$ and $h_{H}$. Similar to \eqref{bc-def-1}, the Bott-Chern form for an invariant polynomial $\Phi$ satisfies the equation
$$\Phi(E,h_{E}) - \Phi(F\oplus H,h_{F}\oplus h_{H})  =\frac{\sqrt{-1}}{2\pi}\,\delb\del\,\widetilde{\Phi}(\cE;h_{E},h_{F},h_{H}),$$
and the functorial property. It vanishes when the exact sequence $\cE$ holomorphically splits and $h_{E}=h_{F}\oplus h_{H}$.                   

For polynomials $\Phi$ corresponding to the Chern character form $\ch$ and total Chern form $c$, we denote, respectively, the corresponding Bott-Chern forms by $\bc$ and $\tilde{c}$.

\begin{remark}
In the smooth manifold case, for linear homotopy of connections $\nabla_{t}$, it possible to integrate over $t$ in the homotopy formula and obtain explicit formulas for the Chern-Simons forms (see, e.g., equation (2.1) in \cite{SS}). In the complex manifold case the situation is more complicated. It is already mentioned in the remark in  \cite[Sect. 3]{BC} that even for a linear homotopy $h_{t}$ of Hermitian metrics, the homotopy
formula in Proposition 3.15 in \cite{BC} contains the inverse metrics through $\Theta_{t}=\delb (h_{t}^{-1}\del h_{t})$, which does not allow to integrate over $t$ in a closed form.
As the result, it is difficult to get explicit formulas for the Bott-Chern forms in terms of Hermitian metrics $h_{1}$ and $h_{2}$ only.
\end{remark}
In \cite{GS}, Gillet and Soul\'{e} gave a construction of the Bott-Chern secondary classes which is also well-suited for short exact sequences of holomorphic vector bundles over $X$, which are used for defining the $K$-theory of $X$. Namely, let $E$ be a holomorphic vector bundle over $X$ with Hermitian metrics
$h_{1}$ and $h_{2}$, let $\mathcal{O}(1)$ be the standard holomorphic line bundle of degree $1$ over the complex projective line $\PP^{1}$, and let $\tilde{E}=E\otimes\mathcal{O}(1)$ be the corresponding vector bundle over $X\times\PP^{1}$.  If $i_{p}:X\rightarrow X\times \C P^{1}$ is the natural inclusion map $i_{p}(x)=(x,p)$ then $i_{p}^{\ast}(\tilde{E})\simeq E$ for all  $p \in \PP^{1}$. Let $\tilde{h}$ be a Hermitian metric on $\tilde{E}$ such that $i_{0}^{\ast}(\tilde{h})=h_{1}$ and $i_{\infty}^{\ast}(\tilde{h})=h_{2}$ (such a metric is constructed using a partition of unity).\\
\indent In this construction, the Bott-Chern secondary form for the Chern character is
\begin{equation} \label{bc-def}
\bc(E;h_{1},h_{2})=\int_{\PP^{1}}\ch(\tilde{E},\tilde{h})\log|z|^{2}
\end{equation}
--- The integral is convergent since $\log|z|^{2}\omega(z)$, where $\omega$ is  any smooth $(1,1)$-form on $\PP^{1}$,  is integrable.

\begin{lemma}[H. Gillet and C. Soul\'{e}] The Bott-Chern form $\bc(E;h_{1},h_{2})$ satisfies equations \eqref{bc-def-1} and \eqref{bc-funct}, and modulo $\im\,\del +\im\,\delb$ does not depend on the choice of Hermitian metric $\tilde{h}$.
\end{lemma}
The proof given in \cite[Section f)]{BGS} uses the current equation
$$ \frac{\sqrt{-1}}{2\pi}\,\delb\del\log|z|^{2} = \delta_{\infty}- \delta_{0},$$
and the proof of Lemma \ref{C-S-1} uses a simplified version of this argument.
As in the previous section, we put
$$\BC(E;h_{1},h_{2})=\bc(E;h_{1},h_{2})\!\!\!\!\!\mod\!\!(\im\,\del +\im\,\delb).$$
\begin{remark} Note that formula \eqref{bc-def} the for Bott-Chern forms uses the Green function $\log|z|^{2}$ of the Laplace operator on $\PP^{1}$, whereas formula \eqref{cs-def} for the Chern-Simons form uses the Green function $g(\theta)$ of the operator $\dfrac{d}{d\theta}$ on $S^{1}$.
\end{remark}

\subsection{Chern forms of trivial bundles}\label{Sub:3.3}
In what follows if $a$ is a nilpotent element of order $r$ of a commutative ring, then by definition $-\log(1-a)=a+\frac{a^2}{2}+\ldots +\frac{a^{r-1}}{r-1}$ and $\frac{1}{1-a}= 1+a+a^2+\ldots+ a^{r-1}$. We start with the following simple linear algebra result.
\begin{lemma} \label{Algebra} Let $A$ be a matrix over $\C$ or over a commutative algebra $\mathcal{A}$ over $\mathbb{C}$, where in the latter case all its matrix elements are nilpotent. Suppose that $A^2=aA$ for some
$a\in\mathcal{A}$, and that $1-\lambda a$ is invertible for all $\lambda$ in some domain $D\subset\C$ containing $0$. Then for such $\lambda$ we have
$$ (I-\lambda A)^{-1} = I+\frac{\lambda}{1-\lambda a} A,$$ and
$$\det(I-\lambda A) = \exp\left\{\frac{\Tr A}{a}\log(1-\lambda a)\right\}.$$ 
In particular, if $\alpha_{i}, \beta_{i}$, $i=1,\dots,k$, are odd elements in some  graded-commutative algebra over $\C$ (e.g., the algebra of complex differential forms on $X$), and $A_{ij}=\alpha_{i}\beta_{j}$, then  $A^2=aA$ where $a=-\Tr A=-\sum_{i=1}^{k}\alpha_{i}\beta_{i}$, and
$$\det(I-\lambda A)=\frac{1}{1-\lambda a}.$$
\end{lemma}
\begin{proof}  For $\lambda\in D$ we have
\begin{equation*} 
(I-\lambda A)^{-1}=I +\frac{\lambda}{1-\lambda a}\,A.
\end{equation*}
To prove the formula for the determinant, we use the identity
$$\frac{d}{d\lambda}\log\det(I-\lambda A)=-\Tr \left\{A(I-\lambda A)^{-1}\right\},\quad\lambda\in D.$$
It is well-known for matrices over $\C$ (and easily proven using the Jordan canonical form), and for matrices with nilpotent entries it easily follows the definition of the determinant. Using formula for the inverse, we obtain
$$\frac{d}{d\lambda}\log\det(I-\lambda A)=-\frac{\Tr A}{1-\lambda a}=\frac{d}{d\lambda}\frac{\Tr A}{a}\log (1-\lambda a),$$
and integrating from $0$ to $\lambda$ using $\det I=1$ gives the result.
\end{proof}
\begin{remark} For matrices over $\C$ equation $A^2=aA$ implies that all eigenvalues of $A$ are either $0$ or $a$, so in this case
$$\det(I-\lambda A)=(1-\lambda a)^m,\quad m\geq 0.$$
\end{remark}
The next result is an explicit computation of the total Chern
form of a trivial vector bundle with a special non-diagonal Hermitian metric.
\begin{lemma} \label{Main} Let $E_{r}=X\times\C^{r}$ be a trivial rank $r$ vector bundle over $X$ with a Hermitian metric $h$ given by
\begin{equation*}
h=h(\s,f_{1},\dots,f_{r-1})=g^{\ast}g, \quad\text{where}\quad g=\begin{pmatrix} 1 & 0 & 0 & \dots &0 & \bar{f}_{1} \\
0 & 1 & 0 & \dots & 0 & \bar{f}_{2} \\
0 & 0 & 1 & \dots & 0 & \bar{f}_{3}\\
\vdots & \vdots & \vdots &\ddots &\vdots &\vdots \\
0 & 0 & 0 & \dots & 1& \bar{f}_{r-1}\\
0 & 0 & 0 & \dots & 0 & e^{\s/2} 
\end{pmatrix},
\end{equation*}
and $f_{1},\dots, f_{r-1}\in C^{\infty}(X,\C)$, $\s\in C^{\infty}(X,\R)$. Then
$$c(E_{r},h)= 
c(E_{1},e^{\s})+ \frac{\sqrt{-1}}{2\pi}\,\delb\del\log\left(1 - \frac{\sqrt{-1}}{2\pi}\,U\right),$$
where $U=e^{-\s}\sum_{i=1}^{r-1} \partial f_i\wedge \delb \bar{f}_{i}$ and $E_{1}=\det E_{r}$ is a trivial line bundle over $X$. Equivalently,
$$c_{1}(E_{r},h)=\frac{\sqrt{-1}}{2\pi}\,\delb\del\s,\;\; c_{k}(E_{r},h)=-\frac{1}{k-1}\left(\frac{\sqrt{-1}}{2\pi}\right)^{k}\!\!\delb\del\,U^{k-1}, \;\; k=2 ,\dots, r.$$
\end{lemma} 
\begin{proof} We compute the total Chern form $c(E_{r},h)$ by a direct calculation which uses Lemma \ref{Algebra}. Let $\Theta=\delb(h^{-1}\del h)$ be the curvature form associated with the Hermitian metric $h$. We need to prove that for every $\lambda\in\C$,
\begin{align*}
\det(I+\lambda\Theta) &=1+\lambda \bar{\partial}\partial \sigma +  \lambda\bar{\partial}\partial \log ( 1 - \lambda U ) \\
& = 1+\lambda\bar{\partial}\partial \sigma - \lambda^2\frac{\delb \del U}{1 - \lambda U} - \lambda^3\frac{\delb U\wedge \del U} {(1- \lambda U)^2},
\end{align*}
where
$$\frac{1}{1-\lambda U}=\sum_{k=0}^{r-1}\lambda^{k}U^{k}\quad\text{and}\quad \frac{1}{(1-\lambda U)^{2}} =\sum_{k=0}^{r-1}(k+1)\lambda^{k}U^{k}.$$
It is convenient to represent the matrix $I + \lambda\Theta$ in the following block form
$$I + \lambda\Theta =\begin{pmatrix} I + \lambda\Theta_{11}& \lambda\Theta_{12} \\
\lambda\Theta_{21} & 1+\lambda\Theta_{22} 
\end{pmatrix},$$
where $(r-1)\times (r-1)$ matrix $\Theta_{11}$, $(r-1)$-vectors $\Theta_{12}, \Theta_{21}^{t}$, and the scalar $\Theta_{22}$ are given by
\begin{gather*}
\Theta_{11}=\left\{-\bar{\partial}(\bar{f}_{i}e^{-\sigma} \partial f_j)\right\}_{i,j=1}^{r-1},\quad \Theta_{12}=\left\{ \bar{\partial}\partial\bar{f}_{i} - \bar{\partial}(\bar{f}_{i}F) -\bar{\partial}(\bar{f}_{i}\partial \sigma)\right\}_{i=1}^{r-1},\\
\Theta_{21}^{t}=\left\{ \bar{\partial}(e^{-\sigma} \partial f_{i}) \right\}_{i=1}^{r-1},\quad \Theta_{22} =\delb\del\sigma +\delb F,
\end{gather*}
and $F=e^{-\sigma} \sum_{l=1}^{r-1}\bar{f}_{l} \partial f_{l}$.
The row operations $R_{i}\mapsto R_{i}+\bar{f}_{i}R_{r}$ transform the matrix 
$I + \lambda\Theta$ to the form
$$\begin{pmatrix} I-\lambda A & b\\
c & d
\end{pmatrix}, $$
where
\begin{gather*}
A =\left\{e^{-\sigma}\delb\bar{f}_{i}\wedge \del f_{j}\right\}_{i,j=1}^{r-1},\quad 
b = \left\{\bar{f}_{i} + \lambda(\bar{\partial}\partial\bar{f}_{i} - \bar{\partial}\bar{f}_{i}\wedge F -\bar{\partial}\bar{f}_{i}\wedge\partial \sigma)\right\}_{i=1}^{r-1},
\end{gather*} 
and we put $c=\lambda\Theta_{21}$, $d=1+\lambda\Theta_{22}$. 

Now it follows from the representation 
$$\begin{pmatrix} I-\lambda A & b\\
c & d
\end{pmatrix} =\begin{pmatrix} I & b\\
c(I-\lambda A)^{-1} & d
\end{pmatrix}\begin{pmatrix} I-\lambda A & 0\\
0 & 1
\end{pmatrix}$$ 
that
$$\det(I+\lambda\Theta)=\det(I-\lambda A)\left(d-c(I-\lambda A)^{-1}b\right),$$
which we compute explicitly using Lemma \ref{Algebra}. Namely,
\begin{gather*}
\det(I+\lambda\Theta)=\frac{1}{1-\lambda U}\Bigg(1+\lambda(\delb\del\sigma +\delb F) \\
-\sum_{i,j=1}^{r-1}\lambda\delb(e^{-\sigma}\del f_{i})\wedge
\Big(\delta_{ij}+\frac{\lambda e^{-\sigma}\delb\bar{f}_{i}\wedge\del f_{j}}{1-\lambda U}\Big)\wedge(\bar{f}_{j} + \lambda(\bar{\partial}\partial\bar{f}_{j} - \bar{\partial}\bar{f}_{j}\wedge(F+\del\sigma))\Bigg).
\end{gather*}
Using equations 
$$\delb F=-U-\delb\sigma\wedge F +e^{-\sigma}\sum_{i=1}^{r-1}\bar{f}_{i}\delb\del f_{i},$$ 
and
\begin{equation*} 
\del U =-\del\sigma\wedge U + \Psi_{+},\quad
\delb U = -\delb\sigma\wedge U + \Psi_{-},
\end{equation*}
\begin{gather*}
\delb\del U= -\delb\del\sigma\wedge U + \delb\sigma\wedge\del\sigma\wedge U +\del\sigma\wedge\Psi_{-} - \delb\sigma\wedge\Psi_{+}
+ \Phi,
\end{gather*}
where
$$\Psi_{+}=e^{-\sigma}\sum_{i=1}^{r-1}\del f_{i}\wedge\delb\del\bar{f}_{i},\;\; 
\Psi_{-}=e^{-\sigma}\sum_{i=1}^{r-1}\delb\del f_{i}\wedge\delb\bar{f}_{i},\;\;
\Phi=e^{-\sigma}\sum_{i=1}^{r-1}\delb\del f_{i}\wedge\delb\del\bar{f}_{i},$$
and simplifying, we obtain
\begin{align*}
\det(I+\lambda\Theta) & =1+\frac{\lambda}{1-\lambda U}\Bigg(\delb\del\sigma -\lambda\Phi +\lambda\delb\sigma\wedge\Psi_{+} - \lambda\delb\sigma\wedge U\wedge F + \lambda\Psi_{-}\wedge F \\ 
& - \lambda\delb\sigma\wedge\del\sigma\wedge U +
\lambda\Psi_{-}\wedge\del\sigma
-\frac{\lambda}{1-\lambda U}\Big(-\delb\sigma\wedge U\wedge F + \Psi_{-}\wedge F \\
& + \lambda\delb\sigma\wedge\del\sigma\wedge U\wedge U 
+ \lambda\Psi_{-}\wedge\Psi_{+} -\lambda\Psi_{-}\wedge U\wedge F -
\lambda\Psi_{-}\wedge U\wedge\del\sigma
\\ & -\lambda\delb\sigma\wedge U\wedge\Psi_{+} + \lambda\delb\sigma\wedge U\wedge U\wedge F\Big)\Bigg) \\
&  =1+\frac{\lambda}{1-\lambda U}\Big(\delb\del\sigma + \lambda(-\Phi +\delb\sigma\wedge\Psi_{+} 
  - \delb\sigma\wedge\del\sigma\wedge U  +
\Psi_{-}\wedge\del\sigma)\Big)   \\ & -\lambda^{3}\frac{\delb U\wedge \del U}{(1-\lambda U)^{2}}  =1+\lambda\delb\del\sigma -\lambda^{2}\frac{\delb\del U}{1-\lambda U} -\lambda^{3}\frac{\delb U\wedge \del U}{(1-\lambda U)^{2}}. 
\end{align*}
\end{proof}
\begin{remark} Following the suggestion of the referee, here we give a more invariant proof of Lemma \ref{Main}. Namely, the  bundle $(E_r, h)$ is a metric extension of the bundle $(E_{r-1},I)$ with the flat metric $I$ by the line bundle $(E_1,e^{\sigma})$ with the metric $e^\sigma$. Dualizing, we obtain the following short exact sequence  
\begin{equation*}
 \begin{CD} 0 @>>> E_{1}^{*} @>i>> E^{*}_r  @>p>> E^{*}_{r-1}  @>>> 0 \end{CD}
\end{equation*}
of Hermitian vector bundles. Applying Proposition \ref{B-C-formula-1} in  Section \ref{Sub:4.2} to this sequence and putting $t=-1$ we immediately get Lemma \ref{Main}. Indeed, in this case $\Theta_{H}=0$ and an easy computation of the second fundamental form of $E_{1}^{*}$ in $E_{r}^{*}$ gives the result. 
\end{remark}
\begin{corollary} \label{Id} The following identities hold for $k=0,\dots,r$,
\begin{gather*} \label{zero}
\displaystyle \sum _{l=0} ^{r} (-1)^l \ch_k (\Lambda ^l E_r^{*}) = -\frac{\delta _{kr}}{r-1} \left ( \frac{\sqrt{-1}}{2\pi} \right ) ^r \bar{\partial}{\partial}\left(U^{r-1}\right) \\
=-(r-2)!\delta _{kr} \left ( \frac{\sqrt{-1}}{2\pi} \right ) ^r \bar{\partial}{\partial}\left(e^{(r-1)\s}\del f_{1}\wedge\delb\bar{f}_{1}\wedge\cdots\wedge\del f_{r-1}\wedge\delb\bar{f}_{r-1}\right).
\end{gather*}
\end{corollary}
\begin{proof} Immediately follows from Lemmas \ref{Hirzerbruch} and  \ref{Main}.
\end{proof}

\begin{remark} For the rank $2$ trivial vector bundle $E_{2}$ with the Hermitian metric
\begin{equation*}
h=h(\sigma,f)=\begin{pmatrix} 1 & \bar{f} \\
f & |f|^{2} +e^{\sigma} \\
\end{pmatrix} = \begin{pmatrix} 1 & 0 \\
f & e^{\sigma/2} \\
\end{pmatrix}\begin{pmatrix} 1 & \bar{f} \\
0 & e^{\sigma/2} \\
\end{pmatrix}.
\end{equation*}
the main identity in the Corollary \ref{Id} takes the form 
\begin{equation*}
\ch_{2}(E_{2},h(\sigma,f))-\ch_{2}(E_{1},e^{\sigma})=-\frac{1}{(2\pi)^{2}}\delb\del(e^{-\s}\del f\wedge \delb\f),
\end{equation*}
and can be verified by a straightforward computation.
For the bundles of ranks $3$ and $4$ corresponding identities were first verified using special Mathematica package for computing Chern character forms, written by Michael Movshev. Here we prove these identities.
\end{remark}
\begin{remark} The Bott-Chern forms have been already used by physicists in their study of supersymmetric quantum field theories. Thus 
setting $\bar{\theta}=h^{-1}\delb h$, it is easy to obtain
$$\Tr(\theta\wedge\bar{\theta})=e^{-\s}\del f\wedge \delb\bar{f} +e^{-\s}\del\bar{f}\wedge\delb f.$$
To get rid of the second term and to write down the simplest nontrivial Bott-Chern form $\bc_{1}(h,I)$, where $I$ is a trivial Hermitian metric on $E_{2}$, we need to add the ``Wess-Zumino term'' (rather its $(1,1)$-component) to the ``kinetic term'' $\Tr(\theta\wedge\bar{\theta})$. Such formula was first obtained by A. Alekseev and S. Shatashvili in \cite{AS}, where for the case of the Minkowski signature the decomposition
\begin{equation*}
\begin{pmatrix} 1 & \bar{f} \\
f & |f|^{2} +e^{\sigma} \\
\end{pmatrix} = \begin{pmatrix} 1 & 0 \\
f & 1 \end{pmatrix} \begin{pmatrix} 1 & 0 \\
0 & e^{\sigma} 
\end{pmatrix} \begin{pmatrix} 1 & \bar{f} \\
0 & 1 \\
\end{pmatrix}
\end{equation*}
is replaced by the Gauss decomposition for $\mathrm{SL}(2,\C)$. 

The Bott-Chern forms (or rather their exponents) also appear quite naturally in supersymmetric quantum field theories as ratios of non-chiral partition functions for the higher dimensional versions of the so-called $bc$-systems \cite{LMNS}.
\end{remark}

\subsection{The main result}\label{Sub:3.4}
The first result is an analogue of Lemma \ref{Connes} for general complex manifolds.
\begin{lemma} \label{Connes-hol} Let $X$ be a complex manifold. Every $\omega\in\A^{k,k}(X,\C)\cap\A^{2k}(X,\R)$ can be written as a finite linear combination over $\R$ of wedge products of real 
$(1,1)$-forms of the type $\sqrt{-1}\,e^{\sigma}\del f\wedge\delb\bar{f}$, where $\sigma\in C^{\infty}(X,\R)$ and $f\in C^{\infty}(X,\C)$. Moreover, if $\omega$ is zero on open $U\subset X$,
than one can choose these forms such that all functions $\s$ and $f$ vanish on $U$. 
\end{lemma}
\begin{proof} 
 Let $\omega$ be a real form of type $(k,k)$. According to Lemma \ref{Connes}, it is a finite sum of the terms
$$h_{0}dh_{1}\wedge dh_{2}\wedge\cdots\wedge dh_{2k}=h_{0}(\partial +\bar{\partial})h_{1}\wedge (\partial +\bar{\partial}) h_{2}\wedge\cdots\wedge (\partial +\bar{\partial}) h_{2k},$$
where $h_{0},\dots,h_{2k}$ are smooth real-valued functions on $X$ with $h_0 >0$. Notice that any smooth function $h_0$ maybe written as the difference of two smooth positive functions $h_0 =e^{(h_0)^2}+h_0 - e^{(h_0)^2}$.
Since $\omega$ is of type $(k,k)$, it is a finite sum of $(k,k)$-components of the forms above. Every such component is a function times the wedge product of the following factors (where $h$ and $g$ are some of the $h_{i}$'s)
$$\partial h\wedge \bar{\partial} g + \bar{\partial} h\wedge \partial g = \sqrt{-1}(\partial f \wedge\bar{\partial}\bar{f}-\partial h\wedge\bar{\partial} h - \partial g\wedge\bar{\partial} g),
$$
where $f=h+\sqrt{-1}g$.
\end{proof}

For compact complex manifolds (or rather for manifolds admitting a finite coordinate open cover) and for submanifolds of $\mathbb{C}^n$, there is a different version of Lemma \ref{Connes-hol}.
\begin{lemma} \label{Connes-hol-2} Let $X$ be a compact complex manifold or
a submanifold of $\C^{n}$. Every $\omega\in\A^{k,k}(X,\C)\cap\A^{2k}(X,\R)$ can be written as a finite linear combination of wedge products of real $(1,1)$-forms of the type $\sqrt{-1}\,h\bar{\partial}{\partial} \rho$ where $h$ and $\rho$ are smooth real functions on $X$. 
\end{lemma}
\begin{proof}
 Let $\{U_{\alpha}\}_{\alpha\in A}$ be a finite coordinate open cover of $X$ and $\{\rho_{\alpha}\}_{\alpha\in A}$ be a partition of unity subordinate to it, so that
$\omega=\sum_{\alpha\in A}\rho_{\alpha}\left.\omega\right|_{U_{\alpha}}$.
Denoting by 
$$z^{1}=x^{1}+\sqrt{-1}\,y^{1},\dots, z^{n}=x^{n}+\sqrt{-1}\,y^{n}$$ 
local complex coordinates in $U_{\alpha}$, we can write 
$$\left.\omega\right|_{U_{\alpha}}=\sum_{I,J}f_{\alpha, IJ}dx^{i_{1}}\wedge\dots\wedge dx^{i_{l}}\wedge dy^{j_{1}}\wedge\dots\wedge dy^{j_{m}},$$
where $I=\{i_{1},\dots,i_{l}\}, J=\{j_{1},\dots,j_{m}\}$, $f_{\alpha, IJ}\in C^{\infty}(U_{\alpha},\R)$  and $1\leq i_{1}<\dots<i_{l}\leq n$, $1\leq j_{1}<\dots<j_{m}\leq n$, $l+m=2k$.
Since the form $\omega$ was supposed to be of $(k,k)$ type, so are the forms $\left.\omega\right|_{U_{\alpha}}$. On the other hand, the $(k,k)$-component of these forms can be obtained by rewriting them in complex coordinates using
	$$dx^{i}=\frac{1}{2}(dz^{i}+d\bar{z}^{i}),\quad dy^{i}=\frac{1}{2\sqrt{-1}}(dz^{i}-d\bar{z}^{i}),\;\;i=1,\dots,n,$$
and collecting terms of the type $(k,k)$. If one of such terms has a factor
$dz^{i}\wedge d\bar{z}^{l}$, $i,k\in I$, then it necessarily has a factor 
$$(dz^{i}\wedge d\bar{z}^{l}+d\bar{z}^{i}\wedge dz^{l}),$$ 
if it comes from $dx^{i}\wedge dx^{l}$. Similarly, one has factors $(dz^{j}\wedge d\bar{z}^{m}+d\bar{z}^{j}\wedge dz^{m})$, $j,m\in J$, coming from $dy^{j}\wedge dy^{m}$, and
$\sqrt{-1}(dz^{i}\wedge d\bar{z}^{j}-d\bar{z}^{i}\wedge dz^{j})$, coming from $dx^{i}\wedge dy^{j}$, $i\in I$ and $j\in J$. 

In the first two cases the corresponding factors can be written as 
$$2\sqrt{-1}\del\delb(\im(z^{i}\bar{z}^{l}))\quad\text{and}\quad 2\sqrt{-1}\del\delb(\im(z^{j}\bar{z}^{m})),$$ 
whereas in the third case it takes the form $2\sqrt{-1}\del\delb(\re(z^{i}\bar{z}^{j}))$. Let $K_{\alpha}$ be a compact set such that $\mathrm{supp}\,\rho_{\alpha}\subsetneqq K_{\alpha}\subset U_{\alpha}$ and let $b_{\alpha}$ be a corresponding ``bump function" --- a smooth function on $X$ which is 1 on $\mathrm{supp}\,\rho_{\alpha}$ and zero outside $K_{\alpha}$. Then we see that all the terms will take the form
$2\sqrt{-1}\del\delb\rho$, where $\rho(z)=\im(b_{\alpha}(z)z^{i}\bar{z}^{j})$ or $\rho(z)=\re(b_{\alpha}(z)z^{i}\bar{z}^{j})$. This proves the first part of the statement.
The second statement of the lemma is obvious from the construction. For submanifolds of $\mathbb{C}^n$, the local part of the argument above carries over globally by extension of the quantities involved to a tubular neighborhood.
\end{proof}

\begin{remark} Let $\omega$ be a real differential form of pure type on a complex manifold $X$, $\omega\in\A^{k,k}(X,\C)\cap\A^{2k}(X,\R)$.
We call the form $\omega$ \emph{elementary}, if it is a wedge product of $(1,1)$-forms $\sqrt{-1}\,h\delb\del\rho$ with real-valued $h$ and $\rho$, and we call the form $\omega$ \emph{composite}, if it  is a wedge product of $(1,1)$-forms $\sqrt{-1}\,e^{\sigma}\del f\wedge\delb\bar{f}$. According to Lemmas
\ref{Connes-hol} and \ref{Connes-hol-2}, on a compact complex manifold $X$
 every composite form is a finite linear combination of elementary forms and conversely, every elementary form is a finite linear combination of composite forms. This is reminiscent of the ``nuclear democracy'' in the bootstrap model of the $S$-matrix theory in particle physics.
\end{remark}  

We have the following complex manifold analogue of Proposition \ref{Venice}.
\begin{theorem} \label{Venice-hol} For every $\delb\del$-exact form $\omega\in\A(X,\C)\cap \A^{\mathrm{even}}(X,\R)$ on a complex manifold $X$ there is a trivial vector bundle $E$ over $X$ with two Hermitian metrics $h_{1}$ and $h_{2}$ such that
$$\ch(E,h_{1})-\ch(E,h_{2})=\omega.$$
\end{theorem}
\begin{proof} 
It is convenient to introduce a  virtual Hermitian bundle $\mathcal{E}=E-E$ with corresponding Hermitian metrics $h_{1}$ and $h_{2}$, and to rewrite the above equation as $\ch\,\mathcal{E}=\omega$. 
The Chern character form defined this way for virtual Hermitian bundles is obviously additive: if $\mathcal{W}_{1}=W_{1}-W_{1}$ with Hermitian metrics $h_{11}$ and  $h_{12}$, and $\mathcal{W}_{2}=W_{2} - W_{2}$ with Hermitian metrics $h_{21}$ and $h_{22}$, then
$$\ch\,\mathcal{W}_{1} + \ch\,\mathcal{W}_{2}=\ch\,\mathcal{W},$$
where $\mathcal{W}=W-W$ and $W  =W_{1}\oplus W_{2}$ with corresponding Hermitian metrics $h_{1}=h_{11}\oplus h_{21}$ and
$h_{2}=h_{12}\oplus h_{22}$. Slightly abusing notations, we will write $\mathcal{W}=\mathcal{W}_{1}\oplus\mathcal{W}_{2}$. The Chern character form for virtual Hermitian bundles is also multiplicative: 
$$\ch\,\mathcal{W}_{1}\,\ch\,\mathcal{W}_{2}=\ch\,\mathcal{W},$$
where $\mathcal{W}=W-W$ and $W  =(W_{1}\otimes W_{2})\oplus (W_{1}\otimes W_{2})$ with corresponding Hermitian metrics
$$h_{1}=(h_{11}\otimes h_{21})\oplus(h_{12}\otimes h_{22})\quad\text{and}\quad
h_{2} =(h_{11}\otimes h_{22})\oplus(h_{12}\otimes h_{21}).$$
Slightly abusing notations, we will write $\mathcal{W}=\mathcal{W}_{1}\otimes\mathcal{W}_{2}$. 

Let $\omega$ be a real form of degree $(k,k)$, $k>1$\footnote{For $k=1$, we may use a trivial line bundle with metric $e^{\sigma}$.} which is a $\delb\del$ of
a composite form:
\begin{equation} \label{omega}
\omega=\!-(k-2)!\left ( \frac{\sqrt{-1}}{2\pi} \right)^{k} \bar{\partial}{\partial}\left(e^{(k-1)\s}\del f_{1}\wedge\delb\bar{f}_{1}\wedge\cdots\wedge\del f_{k-1}\wedge\delb\bar{f}_{k-1}\right)\!.
\end{equation}
It follows from Corollary \ref{Id} that
\begin{equation} \label{Gauss}
\omega=\ch_{k}\mathcal{F}_{k}\quad\text{and}\quad\ch_{i}\mathcal{F}_{k}=0,\quad i=1,\dots,k-1,
\end{equation}
where the virtual bundle $\mathcal{F}_{k}$ is
$$\mathcal{F}_{k}=\bigoplus_{l=0}^{k}(-1)^l \Lambda ^l E^{*}_k$$
with the naturally induced Hermitian metric. 

In particular, if $\omega$ is a composite form of the top degree $(n,n)$, then
$$\omega = \ch\,\mathcal{F}_n.$$
Now, we may use an induction argument to finish the proof. Namely, suppose that
the statement holds for all forms of degrees $(l,l)$, $k<l\leq n$, and let $\omega$ be a composite $(k,k)$-form given by \eqref{omega}. According to \eqref{Gauss},
$\omega- \ch\,\mathcal{F}_{k}$ is a sum of forms of degrees $(l,l)$ with $l>k$,
so that by the induction hypothesis there exists a virtual Hermitian bundle $\mathcal{F}$ such that $\omega-\ch\,\mathcal{F}_{k} = \ch\,\mathcal{F}$. Thus 
\begin{equation*}
\omega = \ch\,\mathcal{E}, \quad\text{where}\quad \mathcal{E}=\mathcal{F}_{k}\oplus \mathcal{F}.
\end{equation*}
For a general $\delb\del$-exact form $\omega$ of degree $(k,k)$ we have 
$$\omega=\omega_{1}+\cdots +\omega_{m}-\omega_{m+1}-\cdots -\omega_{N},$$
where $\omega_{i}$ are composite forms, so that 
$$\omega=\ch\,\mathcal{E}\quad\text{where}\quad \mathcal{E}= (\mathcal{E}_{1}\oplus\cdots\oplus\mathcal{E}_{m})-(\mathcal{E}_{m+1}\oplus\cdots\oplus\mathcal{E}_{N}).$$ 
\end{proof}

\begin{remark} When $X$ is compact or is a submanifold of $\C^{n}$, we can give another proof using Lemma \ref{Connes-hol-2}.
Firstly, the statement holds for $(1,1)$-forms. Namely, since every real $\delb\del$-exact $(1,1)$-form is given by $\omega=\frac{\sqrt{-1}}{2\pi}\,\delb\del\s$, 
where $\s\in C^{\infty}(X,\R)$, consider the trivial holomorphic line bundle $E_{1}$ with the Hermitian metric $h=e^{\s}$, so that  
$$\ch(E_{1},h)=\exp\omega=1+\omega+\frac{1}{2}\omega^{2}+\dots+\frac{1}{n!}\omega^{n}.$$
To get rid of all terms in this expression except $\omega$, consider Hermitian metrics $e^{\alpha_{i}\s}$, $i=1,\dots, n+1$, and choose pairwise distinct $\alpha_{i}$ such that the following system of equations
$$
\begin{pmatrix} 1 & 1 & 1& \dots &1 \\
\alpha_{1} & \alpha_{2} & \alpha_{3} & \dots & \alpha_{n+1} \\
\alpha^{2}_{1} & \alpha^{2}_{2} & \alpha_{3}^{2 } & \dots & \alpha^{2}_{n+1}\\
\vdots & \vdots & \vdots & \vdots & \vdots \\
\alpha^{n}_{1} & \alpha^{n}_{2} & \alpha_{3}^{n } & \dots & \alpha^{n}_{n+1}\\
\end{pmatrix}\begin{pmatrix} r_{1 }\\
r_{2}\\
r_{3} \\
\vdots\\
r_{n+1}
\end{pmatrix}=\begin{pmatrix} 0\\
1\\
0\\
\vdots\\
0
\end{pmatrix}
$$
has an integer solution $r_{1},\dots,r_{n+1}$. Namely, for any choice of $n+1$ different rational numbers $\alpha_{i}$ numbers $r_{i}$ are also rational. If their least common denominator is $N>1$, then for $\beta_{i}=\alpha_{i}/N$ the corresponding solution is integral. Now putting
$$(E,h_{1})=\bigoplus_{r_{i}>0} r_{i}(E_{1},e^{\beta_{i}\s})\quad\text{and}\quad (E,h_{2})=\bigoplus_{r_{i}<0}(-r_{i})(E_{1},e^{\beta_{i}\s}),$$
where $n(L,h)$ stands for the direct sum of $n$ copies of a line bundle $L$ with the Hermitian metric $h$, we get 
$$\ch(E,h_{1})-\ch(E,h_{2})=\omega.$$ 
Now let $\omega$ be a real form of degree $(k,k)$ which is a $\delb\del$ of an
elementary form:
$$\omega= \left ( \frac{\sqrt{-1}}{2\pi} \right)^{k} \bar{\partial}{\partial}\left(\rho_{1}\delb\del\rho_{2}\wedge\cdots\wedge\delb\del\rho_{k}\right)= \left ( \frac{\sqrt{-1}}{2\pi} \right)^{k} \bar{\partial}{\partial}\rho_{1}\wedge\delb\del\rho_{2}\wedge\cdots\wedge\delb\del\rho_{k}.$$
Then 
$$\omega=\ch\,\mathcal{E},\quad\text{where}\quad \mathcal{E}=\mathcal{E}_{1}\otimes\cdots\otimes\mathcal{E}_{k}.$$ 
For general non-elementary forms, the result follows by means of a linear combination.
\end{remark}
\begin{remark} It immediately follows from the second statement of Lemma \ref{Connes-hol} and the proof of Theorem \ref{Venice-hol}, that if form $\omega$ vanishes on open $U\subset X$, then Hermitian metrics $h_{1}$ and $h_{2}$ can be chosen such that $h_{1}=h_{2}=I$ --- identity matrix ---  on $U$.
\end{remark}
\begin{corollary} \label{V-complex} For every $\omega\in\A(X,\C)\cap \A^{\mathrm{even}}(X,\R)$ of degree not greater than $2n-2$, there is a trivial vector bundle $E$ over $X$ with two Hermitian metrics $h_{1}$ and $h_{2}$ such that in $\tilde{\mathcal{A}}(X,\C)$
$$\BC(E;h_{1},h_{2})=\omega.$$
\end{corollary}
\begin{proof} It is analogous to the proof of Corollary \ref{V-real}. Namely, let
$\Omega\in\A(X\times\PP^{1},\C)\cap \A^{\mathrm{even}}(X\times\PP^{1},\R)$ be such that under the inclusion map $i_{p}: X\rightarrow X\times\PP^{1}$ one
has $i_{\infty}^{\ast}(\Omega)=-\omega$ and $i_{0}^{\ast}(\Omega)=0$ in some neighborhood of $0$ in $\PP^{1}$. It follows from Theorem \ref{Venice-hol} that there is a trivial vector bundle $\tilde{E}$ over $X\times\PP^{1}$ with two Hermitian metrics $\tilde{h}_{1}$ and $\tilde{h}_{2}$ such that
$$\frac{\sqrt{-1}}{2\pi}\,\delb\del\Omega =\ch(\tilde{E},\tilde{h}_{1})-\ch(\tilde{E},\tilde{h}_{2}),$$
where the metrics $\tilde{h}_{1}$ and $\tilde{h}_{2}$ can be chosen such that
$i^{\ast}_{0}(\tilde{h}_{1})=i^{\ast}_{0}(\tilde{h}_{2})=I$. Denoting by $E$ a trivial vector bundle over $X$ --- a pullback of $\tilde{E}$ --- and putting $h_{1}=i^{\ast}_{\infty}(\tilde{h}_{1}),  h_{2}=i^{\ast}_{\infty}(\tilde{h}_{2})$, we obtain, modulo
$\im\,\del +\im\,\delb$, 
\begin{align*}
\bc(E;I,h_{1})-\bc(E;I,h_{2}) & = \int_{\PP^{1}}\frac{\sqrt{-1}}{2\pi}\,\delb\del\,\Omega\log|z|^{2} \\
& = \int_{\PP^{1}}\frac{\sqrt{-1}}{2\pi}\,\delb_{z}\del_{z}\Omega\log|z|^{2} \\
& = \int_{\PP^{1}}\frac{\sqrt{-1}}{2\pi}\,\Omega\,\delb_{z}\del_{z}\log|z|^{2} \\
& =i^{\ast}_{\infty}(\Omega)-i_{0}^{\ast}(\Omega)\\
& =-\omega.
\end{align*}
Therefore in $\tilde{\mathcal{A}}(X,\C)$,
$$\omega=-\BC(E; I,h_{1})+\BC(E; I,h_{2})=\BC(E;h_{1},h_{2}). 
$$
\end{proof}
\section{Applications}\label{Sect:4}
In this section we describe some applications of Lemma \ref{Algebra} and Corollary \ref{V-complex}. 
\subsection{Differential $K$-theory} \label{Sub:4.1}
Recall that according to the definition of differential $K$-theory in \cite{GS},  the $K$-group $\hat{K}_{0}(X)$ for a complex manifold $X$ is defined as the free abelian group generated by the triples $(E,h,\eta)$,  where $E$ is a holomorphic vector bundle over $X$ with Hermitian metric $h$ and $\eta\in\tilde{A}(X,\C)$, with the following relations. For every exact sequence $\cE$ 
\begin{equation*}
 \begin{CD} 0 @>>> F @>i>> E  @>p>> H  @>>> 0 \end{CD}
 \end{equation*}
of holomorphic vector bundles over $X$, endowed with arbitrary Hermitian metrics $h_{F}, h_{E}$ and $h_{H}$, impose 
\begin{equation} \label{gs-rel}
(F,h_{F},\eta') + (H,h_{H},\eta'')=(E,h_{E},\eta'+\eta'' +\BC(\cE,h_{E},h_{F},h_{H})),
\end{equation}
where $\BC(\cE,h_{E},h_{F},h_{H})=\bc(\cE,h_{E},h_{F},h_{H})\!\!\!\mod\!(\im\,\del +\im\,\delb)$. It follows from \eqref{gs-rel} that in $\hat{K}_{0}(X)$
\begin{equation} \label{well-def}
(E,h_{1},\eta_{1})=(E,h_{2},\eta_{1}+\BC(E,h_{1},h_{2})).
\end{equation}
Now following \cite{SS}, we define two Hermitian metrics $h_{1}$ and $h_{2}$ on the holomorphic vector bundle $E$ to be equivalent, if $\BC(E,h_{1},h_{2})=0$, and define a \emph{structured} holomorphic Hermitian vector bundle $\mathcal{E}$ as a pair $(E,\{h\})$, where $\{h\}$ is the equivalence class of a Hermitian metric $h$. Our goal is to impose relations on the free abelian group generated by $\mathcal{E}$ such that the resulting group $H\hat{K}_{0}(X)$ is isomorphic to the ``reduced'' differential $K$-theory group $\hat{K}_{0}^{\mathrm{rd}}(X)$ - a subgroup of $\hat{K}_{0}(X)$ with forms $\eta$ of degrees not greater than $2n-2$.

	First we observe that it follows from \eqref{well-def} that the mapping  
\begin{equation} \label{map}
\mathcal{E}=(E,\{h\}) \mapsto \varepsilon(\mathcal{E})=(E,h,0)\in \hat{K}_{0}^{\mathrm{rd}}(X)
\end{equation}
is well-defined. Next we show that when extended to to the free abelian group generated by the structured holomorphic Hermitian bundles, this mapping is onto.
Indeed, for every $\eta\in\tilde{A}(X,\C)\cap A^{\mathrm{even}}(X,\R)$ of degree not greater than $2n-2$, let $F$ be the trivial vector bundle over $X$ with two Hermitian metrics $h_{1}$ and $h_{2}$ such that, according to Corollary \ref{V-complex},
$$(F,h_{1},\eta)=(F,h_{2},0).$$
in $\hat{K}_{0}(X)$. Since
\begin{align*}
(E\oplus F,h\oplus h_{1},\eta) & = (E,h,0) + (F,h_{1},\eta) \\
& =(E,h, \eta) + (F,h_{1},0),
\end{align*}
we obtain
$$(E,h,\eta)=(E,h,0) +(F,h_{2},0) - (F,h_{1},0).$$
Finally, we define the group $H\hat{K}_{0}(X)$ as the quotient of the free abelian group generated by $\mathcal{E}$ modulo the relations --- pullbacks of the defining relations for $\hat{K}_{0}^{\mathrm{rd}}(X)$ by the mapping $\varepsilon$.
Explicitly, for every exact sequence $\cE$ of holomorphic vector bundles over $X$ with Hermitian metrics $h_{F}, h_{E}$ and $h_{H}$ satisfying $\BC(\cE; h_{E},h_{F},h_{H})=0$, we impose 
$$(F,\{h_{F}\})+(H,\{h_{H}\})=(E,\{h_{E}\}).$$
\begin{remark}
The construction of the group $H\hat{K}_{0}(X)$ is the first step in defining a `Simons-Sullivan model' of the differential $K$-theory of Gillet and Soul\'e. The main open problem is to describe the group operation directly in terms of the structured bundles $\mathcal{E}$. The group structure of the Simons-Sullivan model $\Str(X)$ is given by direct sums \cite{SS}, which is no longer true for the group $H\hat{K}_{0}(X)$. This is because in general short exact sequences $\cE$  do not split holomorphically.  Even in the case when they do, the issue of finding ``structured inverses" remains a challenging problem simply because the proof that works in the smooth context breaks down for fundamental reasons. One of the issues has to do with finding metrics whose Chern forms are of a certain type. (See \cite{Vam} for the discussion of a related question.)
\end{remark}
\subsection{Bott-Chern forms for short exact sequences} \label{Sub:4.2}
Here we extend the computations in \cite[Sect. 4]{BC} and present an
explicit formula for the Bott-Chern form for a short exact sequence $\cE$
\begin{equation*}
 \begin{CD} 0 @>>> F @>i>> E  @>p>> H  @>>> 0 \end{CD}
 \end{equation*}
of holomorphic vector bundles over $X$ in the case when $F$ is a line bundle. Specifically, we consider the case when a Hermitian metric $h$ on $E$ 
defines a metric $h_F$ on $F$ by the restriction on $i(F)$, and a metric $h_H$ on $H$ --- by the $C^\infty$ isomorphism between $H$ and the orthogonal complement $i(F)^{\perp}$ of $i(F)$ in $E$. In other words, there is a $C^\infty$ isometric isomorphism 
\begin{equation} \label{isom}
f: E\simeq F\oplus H, \quad f=i^{\ast}\oplus p,
\end{equation}
where $\ast$ denotes the adjoint map with respect to given metrics. The inverse map is given by $f^{-1}=i+p^{\ast}$. We have
$i^{\ast} i=I_F$, $p p^\ast=I_H$ --- corresponding identity maps in $F$ and $H$, and $i i^{\ast} =P_F$, $p^\ast p=P_H$ --- corresponding orthogonal projections from $E$ onto $i(F)$ and $i(F)^\perp$. The canonical connection $\nabla$ on $(E,h)$
gives rise, respectively, to the canonical connections $\nabla_{F}=P_{F}\circ\nabla\circ P_{F}$ and $\nabla_{H}=P_{H}\circ\nabla \circ P_{H}$ on $(F,h_{F})$ and $(H,h_{H})$ with the curvatures $\Theta_{F}$ and $\Theta_{H}$.
Denoting by $\left.\nabla\right|_{F}=\nabla\circ P_{F}$ restriction on $\nabla$ to $F\subset E$ using \eqref{isom}, we get
$$\nabla_{F} =  \left.\nabla\right|_{F} - A,$$
where $A= -P_{F}\nabla(P_{F}) + \nabla(P_{F})=P_{H}\nabla(P_{F})$ is a $(1,0)$-form with values in $\mathrm{End}(F,H)$, called the second fundamental form of $F$ in $E$ (see, e.g., \cite[p. 72]{GH}). Correspondingly, $A^{\ast}=-P_{F}\nabla(P_{H})$ is a $(0,1)$-form with values in $\mathrm{End}(H,F)$ and 
\begin{align*}
\Theta_{F}=\left.\Theta\right|_{F} + A^{\ast}\wedge A \\
\Theta_{H}=\left.\Theta\right|_{H} + A\wedge A^{\ast} 
\end{align*}
(see, e.g., \cite[p. 78]{GH}). Under the isomorphism \eqref{isom}, $\left.\Theta\right|_{F} =i^{\ast}\,\Theta\, i$ and $\left.\Theta\right|_{H} =p\,\Theta\, p^\ast$.

To compute 
\begin{equation*} 
\tilde{c}(\cE;h)=\tilde{c}(\cE;h,h_{F},h_{H}),
\end{equation*}
In their paper \cite{BC}, Bott and Chern introduced a linear homotopy of connections on $E$
$$\nabla_{u}=\nabla + (u-1)A,\quad 0 \leq u\leq 1,$$
and in \cite[Lemma 4.8]{BC} explicitly computed its curvature $\Theta(u)=\nabla_{u}^{2}$. Using the isomorphism \eqref{isom} it is given
by the following $2\times 2$ block-matrix\footnote{To compare with notations in \cite{BC}, $A=\delta$ and $u=e^{t}$. }
\begin{equation} \label{curv-BC}
\Theta(u) =\begin{bmatrix} \Theta_{F} - u A^{\ast}\wedge A & i^{\ast}\,\Theta\, p^{\ast}  \\
u p\,\Theta \,i & \Theta_{H} - u A\wedge A^{\ast}
\end{bmatrix}
\end{equation}
Since
\begin{equation*}
\Theta(0) =\begin{bmatrix} \Theta_{F}  & i^{\ast}\,\Theta\, p^{\ast}  \\
0 & \Theta_{H} 
\end{bmatrix},
\end{equation*}
and $\Theta(1)=\Theta$, The Bott-Chern homotopy formula \cite[ Eqn. (4.13)]{BC} gives\footnote{The subtraction of $D(0)$ is equivalent to omitting the term $a_0$ in formula (4.18) in \cite{BC}.}
\begin{equation}\label{B-C-1}
\tilde{c}(\cE,h)=\int_{0}^{1}\frac{D(u)-D(0)}{u}du,
\end{equation}
where $D(u)$ -- is the linear in $\lambda$ term in $ \det(\kappa\Theta(u)+\lambda P_{F})$, and $\kappa=\dfrac{\sqrt{-1}}{2\pi}$. This is the main result of Section 4 of \cite{BC}.

It is convenient to introduce generating functions $c_t(E,h)=\sum_{k=0}^r t^k c_k(E,h)$ and similar for the bundle $(F\oplus H,h_F\oplus h_H)$, so that
$$c_t(E,h) -c_t(F\oplus H,h_F\oplus h_H) =\frac{\sqrt{-1}}{2\pi}\,\delb\del  \tilde{c}_t(\cE,h),$$
where
\begin{equation} \label{BC-generate}
\tilde{c}_t(\cE;h)=\sum_{k=1}^r t^k  \tilde{c}_k(\cE;h,h_F,h_H),
\end{equation} 
since $\tilde{c}_0(\cE,h)=0$. In terms of generating functions \eqref{B-C-1} takes the form
\begin{equation}\label{B-C-2}
\tilde{c}_t(\cE,h)=\int_{0}^{1}\frac{D_t(u)-D_t(0)}{u}du,
\end{equation}
where
$$D_t(u)=\sum_{k=1}^r t^k D_k(u),$$
and $D_{k}(u)$ is linear in $\lambda$ term in $\Tr\Lambda^{k}(\kappa\Theta(u)+\lambda P_F)$. 
Explicitly,
\begin{equation} \label{polarize}
D_{k}(u)=\sum_{l=1}^{k}\Tr \kappa\Theta(u)\wedge\cdots \kappa\Theta(u)\wedge P_{F} \wedge\kappa\Theta(u)\cdots\wedge\kappa\Theta(u),
\end{equation}
where in each term $P_{F}$ appears at the $l$-th place in the $k$-fold wedge product. 

Formula \eqref{B-C-2} was re-derived by C. Mourougane \cite{Mou}. It was observed there that
when $F$ is a line bundle,
$$D_{k}(u)=\Tr \Lambda^{k-1}\left(\kappa p\,\Theta(u) \,p^{\ast}\right).$$
which can be easily seen by evaluating \eqref{polarize} in a local unitary frame $e_1,\dots,e_r$ of the bundle $E$ over $U\subset X$ such that under the isomorphism \eqref{isom} $e_1$ spans $F$, and
$e_2,\dots,e_r$ span $H$.
As it follows from \eqref{curv-BC}, the corresponding generating function takes the form 
$$D_{t}(u)=\sum_{k=1}^{r} D_{k}(u)t^{k}=\det\left(I +\tilde{t}\left(\Theta_{H}-u A \wedge A^{\ast}\right)\right),$$
where $\tilde{t}=\kappa t$.
Now we can apply Lemma \ref{Algebra} and express $\tilde{c}_t(\cE,h)$ in a closed form, using only the data given by the isomorphism \eqref{isom}. 
\begin{proposition} \label{B-C-formula-1} Let $\cE$ be a short exact sequence 
\begin{equation*}
 \begin{CD} 0 @>>> F @>i>> E  @>p>> H  @>>> 0 \end{CD}
 \end{equation*}
of holomorphic vector bundles over $X$, equipped with Hermitian metrics $h_F, h$ and $h_H$, where the metric $h_F$ on $F$ is the restriction of the metric $h$ on $i(F)\subset E$, and the metric $h_H$ on $H$ is defined by the $C^\infty$ isomorphism between $H$ and the orthogonal complement $i(F)^{\perp}$ of $i(F)$ in $E$. Let $A$ be the second fundamental form of $i(F)\subset E$. In the case when $F$ is a line bundle, the generating function for the Bott-Chern forms $ \tilde{c}_t(\cE;h)$, defined by \eqref{BC-generate}, is given by the following explicit formula
$$\tilde{c}_t(\cE;h)=-c_t(H,h_H)\log\left(1+\frac{\sqrt{-1}}{2\pi}  \left\{\Tr \left(I+\frac{\sqrt{-1}}{2\pi} t\Theta_{H}\right)^{-1}\!\!t \,A\wedge A^{\ast}\right\}\right).$$

\end{proposition}
\begin{proof}
We have
$$\det(I +\tilde{t}(\Theta_{H}-u A \wedge A^{\ast}))=\det(I+\tilde{t}\Theta_{H})\det(I-u\tilde{t}(I+\tilde{t}\Theta_{H})^{-1}A\wedge A^{\ast}).$$
To apply Lemma \ref{Algebra}, we use the local unitary frame $e_{1},\dots,e_{r}$  described above.  In this frame $\mathrm{End}(F,H)$-valued $(1,0)$-form $A$ is given by $A(e_{1})=a_{2}e_{2}+\cdots +a_{r}e_{r}$, and $\mathrm{End}(H,F)$ valued $(0,1)$-form $A^{\ast}$ --- by $A^{\ast}(e_{i})=\bar{a}_{i}e_{1}$, where $a_{i}\in\mathcal{A}^{1,0}(U)$ and $\bar{a}_{i}\in\mathcal{A}^{0,1}(U)$, $i=2,\dots,r$. Thus $A\wedge A^{\ast}$ is represented by the  $(r-1)\times(r-1)$ matrix of $(1,1)$-forms on $U$ with matrix elements $a_{i}\wedge\bar{a}_{j}$, $i,j=2,\dots,r$.
Denoting corresponding matrix elements of $(I+\tilde{t}\Theta_{H})^{-1}$ by $\theta_{ij}$, $i,j=2,\dots,r$, we get
$$\det(I-u\tilde{t}(I+\tilde{t}\Theta_{H})^{-1}A\wedge A^{\ast})=\det(I- u\tilde{t} B),\quad B_{ij}=\sum_{k=2}^{r}\theta_{ik}a_{k}\wedge \bar{a}_{j}.$$
Since $B^2 = b B$, where $b=-\Tr B=-\Tr (I+\tilde{t}\Theta_{H})^{-1}A\wedge A^{\ast}$, using Lemma \ref{Algebra} we get
\begin{equation} \label{bB}
\det(I-u\tilde{t}(I+\tilde{t}\Theta_{H})^{-1}A\wedge A^{\ast}) =\frac{1}{1+ u\tilde{t} \Tr \{(I+\tilde{t}\Theta_{H})^{-1}A\wedge A^{\ast}\}}.
\end{equation}
Substituting \eqref{bB} into \eqref{B-C-2} and integrating, we get the result.
\end{proof}

As an application of Proposition \ref{B-C-formula-1}, one can easily compute the Bott-Chern forms of the metrized relative Euler sequence, originally derived via a rather intricate combinatorial analysis by C. Mourougane (Theorem $1$ in \cite{Mou}). We leave the details to the interested reader. A link between this formula and the theory of Gillet-Soul\'e comes in the form of Arakelov geometry and the computation of Bott-Chern forms on flag manifolds \cite{Tam}. This shall be explored in another paper.

\subsection*{Acknowledgements} The work of L.T. was partially supported by the NSF grants DMS-0705263 and DMS-1005769. He thanks James Simons and  Dennis Sullivan for carefully explaining their work, and expresses his gratitude to Michael Movshev for writing the Mathematica packages used at the early stage of this work, and for sharing his insight in formulating Lemma \ref{Main}. L.T. is also grateful to Nikita Nekrasov for the suggestion to consider a $(1,1)$ component of the Wess-Zumino-Novikov-Witten $2$-form. We also thank the anonymous referee for constructive remarks and suggestions.
\bibliographystyle{amsalpha}
\bibliography{Bott-Chern-final}
\end{document}